\documentclass[12pt]{article}   	
\usepackage[margin=2.5cm]{geometry}                		
\usepackage{graphicx}
\usepackage{amsmath,amssymb,mathrsfs,amsthm}
\usepackage{mathtools}
\newcommand{\mathsym}[1]{{}}
\newcommand{\unicode}[1]{{}}
\usepackage{soul}
\usepackage{mathtools}
\usepackage{amsfonts}
\usepackage{subcaption}
\usepackage{cleveref}
\usepackage{stackrel}
\theoremstyle{definition}
\usepackage{float}
\usepackage{mathbbol}
\newtheorem{theorem}{Theorem}
\newtheorem{proposition}{Proposition}
\newtheorem{corollary}{Corollary}
\newtheorem{lemma}{Lemma}

\newtheorem{assumptionex}{Assumption}
\newenvironment{assumption}
  {\pushQED{\qed}\assumptionex}
  {\popQED\endassumptionex}
  \newtheoremstyle{mythm}%
{3pt}
{3pt}
{}
{}
{\bfseries}
{}
{.5em}
{}%

\theoremstyle{mythm}
 \newtheorem{remarkex}{Remark}
\newenvironment{remark}
  {\pushQED{\qed}\remarkex}
  {\popQED\endremarkex}

\usepackage{color}

\newcommand{\lrp}[1]{\left(#1\right)}
\newcommand{\lrb}[1]{\left[#1\right]}
\newcommand{\lrc}[1]{\left\{#1\right\}}
\newcommand{\lrl}[1]{\langle#1\rangle}
\renewcommand{\exp}[1]{\textrm{exp}\lrc{#1}}
\renewcommand{\log}[1]{\textrm{log}\lrp{#1}}

\renewcommand{\d}{\text{d}}

\newcommand{\Prob}[1]{\textrm{Pr}\lrp{#1}}

\renewcommand{\lim}{\stackbin[n\to\infty]{}{\text{lim}}}

\newcommand{\F}{|F|}

\newcommand{\A}{|A|}
\newcommand{\T}{|T|}

\newcommand{\ff}[1]{\frac{#1}{2}}
\title{Revisiting High Dimensional Bayesian Model Selection for Gaussian Regression}
\author{Zikun Yang, Andrew Womack}
\date{}							
\begin{document}
\maketitle
\begin{abstract}
Model selection for regression problems with an increasing number of covariates continues to be an important problem both theoretically and in applications. Model selection consistency and mean structure reconstruction depend on the interplay between the Bayes factor learning rate and the penalization on model complexity. In this work, we present results for the Zellner-Siow prior for regression coefficients paired with a Poisson prior for model complexity. We show that model selection consistency restricts the dimension of the true model from increasing too quickly. Further, we show that the additional contribution to the mean structure from new covariates must be large enough to overcome the complexity penalty. The average Bayes factors for different sets of models involves random variables over the choices of columns from the design matrix. We show that a large class these random variables have no moments asymptotically and need to be analyzed using stable laws. We derive the domain of attraction for these random variables and obtain conditions on the design matrix that provide for the control of false discoveries.
\end{abstract}
\section{Introduction}
We are considering a generic Gaussian regression problem with a response vector $ {y}\in\mathbb{R}^n$ and a design matrix $X\in\mathbb{R}^{n\times p}$, such that 
\begin{equation}
\label{eq::the_model}
\begin{array}{rcl}
 {y}=\alpha \mathbb{1}+ {X\beta}+ {\epsilon}\text{, where $ {\epsilon}\sim N(0,\sigma^2 {I}_{n\times n})$}.
\end{array}
\end{equation}
In \cref{eq::the_model}, $\mathbb{1}$ is a $n$-dimensional vector of 1's, $\alpha$  is an intercept term included in all the models, and $\beta\in \mathbb{R}^{p}$ is a vector of regression coefficients. We assume that the data are generated by the true model $M_T$. The dimensions of the true model and the full model are assumed to be growing with the sample size, hence the problem is in the high-dimensional regime. To address the problem,  we also make a sparse assumption such that a large portion of the coefficients in the full model are actually zero, and the goal here is to recover the exact support set,
\begin{equation}
\label{df::true_model_set}
\begin{array}{rcl}
T&:=&\lrc{i\in\lrc{1,\ldots,p}|\beta_i\neq0},
\end{array}
\end{equation}
of the true model. \\

The high-dimensional sparse solutions for the Gaussian regression has been extensively investigated, and one of the most famous and inspiring methods is the Least Absolute Shrinkage and Selection Operator, so-called the Lasso \cite{tibshirani1996regression}. It is well-known that  the estimation of the Lasso achieves sparsity while suffers from a consistent bias, and there has been tremendous efforts for fixing the inadequacies of the Lasso, such as, the Elastic net \cite{zou2005regularization}, the SCAD \cite{fan2001variable}, etc. Meanwhile, Bayesian variants based on the idea of shrinking  coefficients have also received great attentions, e.g., the Bayesian Lasso \cite{park2008bayesian}, the Horseshoe prior \cite{carvalho2009handling} etc. All of these  methods attain the sparse solution by shrinking certain coefficients to zero, and are assessed by some measurements of predictions. Rather than focusing on the criteria based on various $\mathit{L}_q$ norms $\mathbb{E}||\hat{\beta}-\beta_T ||_q^q$, typically $L_1$ and $L_2$ norms, the current paper pays attentions to variable-selection consistency under 0-1 loss through Bayesian methodology. \\

The consistency of Bayesian model selection has been established in many papers. The general procedure of Bayesian model selection is usually based on the posterior evidence provided by  the Bayes factor, which makes a direct comparison between the marginals of the null and the alternative models  \cite{kass1995bayes}. The mathematical nature of the Bayes factor prevents the deployment of the improper prior, since the nuisance constant in the priors would cause the identifiable problem for any statistical inference. To fix this problem, the intrinsic prior was deliberately designed \cite{berger1996intrinsic} to  cancel out the nuisance constants, and the consistency of model selection for pair-wise comparison between the null and alternative model has been established \cite{casella2006objective}, and for the situations where the dimensions of the models are growing with the sample size \cite{moreno2010consistency}. However, the pair-wise consistency can not guarantee the posterior probability of the true model going to one due to the massive number of the models. \\

Another choice of prior distribution for model selection is famous Zellner's $g$-prior given by, 
\begin{equation}
\begin{array}{rcl}
\beta|\sigma^2,M_A&\sim&\text{N}\lrp{0,g\sigma^2\lrp{X_A'X_A}^{-1}},
\end{array}
\end{equation}
where $A\in\lrc{1,\ldots,p}$ and $X_A$ is the corresponding design matrix composed with the predictors whose indexes are in the set $A$.  The popularity of Zellner's $g$-prior is primarily due to its feasibility of the  computation of the corresponding Bayes factor with a closed form. However, this relatively simple form of Bayes factor  arises several paradoxes of model selection , and many empirical choices for `g' parameter have been proposed to remedy the situation \cite{liang2012mixtures}. Liang et al., \cite{liang2012mixtures} proposed the mixtures of $g$-prior by randomizing  `g' parameter to overcome the paradoxes, and showed the consistency of model selection with the mixtures of $g$-prior, especially with Zellner\&Siow prior \cite{zellner1980posterior}. It is also interested to notice that the intrinsic prior can be represented as the mixture of $g$-prior with a proven consistency of model selection \cite{womack2014inference}.\\

One vital assumption in  these papers is that the dimension of the full model is not growing with the sample size, hence the total number of the models is well controlled. Whether or not the size of the true model is growing also has important consequence on the assumptions of the design matrix and the consistency of model selection. With a fixed true model, the procedures, i.e.,   \cite{narisetty2014bayesian} \cite{womack2015model}, can tolerate more crucial conditions, such as the size of the full model can be growing at an exponential rate of the sample size or the eigenvalues of the gram matrix can converge to zero. On the other hand, the conditions or assumptions would be harsher, if the dimension of the true model is growing, especially when the situation of the sparsity is close to linear sparsity  \cite{wainwright2009information}. Also, when the dimensions are growing, the learning rate of the Bayes factor may break down, e.g., the models nesting the true model become indistinguishable to the true model.  Hence, the prior on the model space is often required to offer extra penalties of the dimensionality. These priors, e.g., the sparsity prior  \cite{castillo2015bayesian} \cite{yang2016computational} or the truncated Poisson prior  \cite{womack2015model}, are all designed to make penalization on the overfitted models and control the model size for models that are not identifiable from the data \cite{li2018mixtures}.\\

In this paper, we show that the model selection consistency through the truncated Poisson prior on the model space and the modified Zellner\&Siow prior on the regression parameters. The paper is organized as follows. In Section 2, the main results is presented. We review the assumptions of the design matrix used in other papers, list the assumptions that we adopt, and emphasize the corresponding impacts on the consistency. The difference between the sparsity prior and the truncated Poisson prior is highlighted. In Section 3, we show an interesting utility of the stable law and the domain of attraction on the overfitted models. Section 4 contains the conclusion and possible directions of future works. The proofs of the theorems are in Section 5.   

\subsection{Notification}
As stated above, let $A$, $T$, and $F$ denote the indices set of the predictors in the testing model, the true model, and the full model respectively. For a vector $\beta\in\mathbb{R}^p$ and a set $A\in\lrc{1,\ldots,p}$, $\beta_A$ is the vector of $\beta_i$s for $i\in A$, and $|A|$ is the cardinality of $A$. Let $M_0$ denote the null model, the model only contains the intercept term. 

\section{Main results}
\subsection{Mixture of $g$-prior \& the Bayes factor}
Now we specify the prior distributions of the parameters. For any model $M_A$, the model and the priors are
\begin{equation}
\begin{array}{rcl}
 {y}&=&\alpha \mathbb{1}+ {X\beta}+ {\epsilon}\\
\epsilon&\sim &N(0,\sigma^2\mathit{I}_{n\times n})\\
\alpha,\sigma^2&\sim& \frac{1}{\sigma^2}\\
\beta_A&\sim& Normal(0,\frac{\sigma^2}{\omega}\lrp{\tilde{X}'_A\tilde{X}_A}^{-1})\\
\omega&\sim& \pi(\omega),
\end{array}
\end{equation}
 where $\omega=\frac{1}{g}$, $\tilde{X}'_A\tilde{X}_A=\frac{1}{n}X_A'X_A$, which corresponds to the recommendation of unit information prior  \cite{kass1995bayes}. Zellner\&Siow \cite{zellner1980posterior} place a multivariate Cauchy distribution on the coefficient vector, which can be represented as a mixture of normal distributions with the mixing parameter $\omega\sim\text{Gamma}\lrp{\frac{1}{2},\frac{1}{2}}$.  The Zellner-Siow prior is a multivariate extension of Jeffreys's idea on the normal mean hypothesis problem, where Jeffrey argued that Cauchy prior is the simplest form to fix the paradoxes mentioned above. \\
 
 The goal of this paper is to show the selection consistency, such that 
\begin{equation}
\label{eq::goal}
\begin{array}{rcl}
\stackbin[n\to\infty]{}{\text{lim}}\text{Pr}\lrp{M_T|y}&=&1,
\end{array}
\end{equation}
which is equivalent to show that 
\begin{equation}
\label{eq::summation}
\begin{array}{rcl}
\stackbin[n\to\infty]{}{\text{lim}}\stackbin[M_A\neq M_T]{}{\sum}\frac{m\lrp{y|M_A}}{m\lrp{y|M_T}}\frac{\pi\lrp{M_A}}{\pi\lrp{M_T}}&=&0.
\end{array}
\end{equation}
 The Bayes factor is defined as the ratio of the marginal densities between two competing models in \cref{eq::summation}, such as
 \begin{equation}
\begin{array}{rcl}
BF_{A:T}=\frac{m\lrp{y|M_A}}{m\lrp{y|M_T}},
\end{array}
\end{equation}
where $m\lrp{y|M_A}$ is the marginal density of the model $M_A$ after integrating out the regression parameters. 
 For simple computation, it is worth noting that the marginal density $m\lrp{y|M_A}$ is invariant if switching $\tilde{X}'_A\tilde{X}_A$ by $X_A'X_A$ and assigning $\text{Gamma}\lrp{\frac{1}{2},\frac{n}{2}}$ to $\omega$.  
  Integrating out the parameters $\alpha$, $\beta_A$, and $\sigma^2$ is straight forward, and the Bayes factor can be represented as an univariate integration of $\omega$, such as
 \begin{equation}
 \label{eq::bayes_int}
\begin{array}{rcl}
BF_{A:0}&=&\int_0^\infty\lrp{1+\omega}^{\frac{n-\A}{2}}\omega^{\ff{\A-1}}\lrp{1-R_A^2+\omega}^{-\ff{n-1}}\pi(\omega)\d \omega,
\end{array}
\end{equation}
where $R_A^2=\frac{y'\lrp{H_A-H_1}y}{y'(I-H_1)y}$ is the coefficient of determination and $H_A$ is the projection matrix generated by the design matrix $X_A$. 
In  \cite{liang2012mixtures}, the authors suggested to approximate the integral in \cref{eq::bayes_int} by the Laplace method  \cite{tierney1986accurate}, which requires to solve a cubic function of $\omega$. Since the dimension of the true model is assumed to be fixed in  \cite{liang2012mixtures}, the solution is asymptotically stable as $n\to\infty$, which is not the case in the current set-up. Actually, the unsatisfied approximation is the main reason that the Z\&S prior never got popular in the first place, even with appealing statistical properties.  \\

In this paper, we choose the Beta-prime distribution with the shape parameters $\frac{1}{2}$ and $\ff{n-\A-1}$ as a modified version of the Z\&S prior. The parameters are deliberately chosen to cancel out the term $\lrp{1+\omega}^{\frac{n-\A}{2}}$ in \cref{eq::bayes_int}, also to maintain an asymptotically equivalent behavior of the density function to the Zellner\&Siow prior at the tail area and the origin. The same set-up has been adopted in  \cite{maruyama2011fully}, and the resulting Bayes factor between the testing model $M_A$ and the true model $M_T$ can be shown as
\begin{equation}
\label{eq::bayes_factor}
\begin{array}{rcl}
BF_{A:T}&=&\frac{\lrp{1-R_T^2}^{{\ff{n-\T-1}}}}{\lrp{1-R_A^2}^{\ff{n-\A-1}}}\frac{\Gamma\lrp{\ff{\A}}\Gamma\lrp{\ff{n-\A}} }{\Gamma\lrp{\ff{\T}}\Gamma\lrp{\ff{n-\T}}},
\end{array}
\end{equation}
where $\Gamma(\bullet)$ is the Gamma function. 
\begin{remark}
The major difference between the current paper and  \cite{yang2016computational} is the choice of the original $g$-prior or the mixture of $g$-prior. Certainly, the original $g$-prior with an empirical choice on $g$, e.g., $g=\text{max}\lrc{n,p^2}$ or $g=p^{2\alpha}$ \cite{yang2016computational}  \cite{kass1995reference}, does exhibit certain flexibility for the model selection. However, the information paradox associated with the original $g$-prior demands more assumptions on the signals to prove the consisitency \cite{liang2012mixtures} \cite{shang2011consistency} \cite{yang2016computational}. Also, when the null model is the true model, consistency
only holds true for the Zellner-Siow prior, but does not hold for the empirical $g$-prior \cite{liang2012mixtures}.  The advantages of adopting the mixture of $g$-prior are easy to observe. First, it avoids the tuning step of $g$, since it has been integrate out for the marginals. Second, it can be shown that either when $M_T=M_N$ or $M_T\neq M_N$, the Bayes factor associated with the mixture of $g$ holds the consistency of the selection, where it is not the case for the empirical $g$-prior. 
\end{remark}

\subsection{Model prior}
In the fixed dimension regime, the prior of models is usually given equal prior probability and would be canceled out during the procedure. Also, the convergence of the pair-wise comparison  is enough for establishing the consistency, since the summation in \cref{eq::summation} is only over a finite number of models. However, things are different in the high-dimensional regime, and the convergence of the pair-wise comparison doesn't ensure the posterior selection consistency in \cref{eq::goal}. \\

For a generic model $M_A$, the prior probability of choosing $M_A$ can be represented as
\begin{equation}
\begin{array}{rcl}
\pi(M_A)&=&\pi(\A){p\choose \A}^{-1},
\end{array}
\end{equation}
where $\pi(\A)$ is the prior probability of the size $\A$, and all the models with same size equally share the same probability mass $\pi(\A)$. In  \cite{castillo2015bayesian}  \cite{narisetty2014bayesian}  \cite{yang2016computational}, the authors chose the sparse prior  to introduce strong penalty on the dimension  with the form
\begin{equation}
\label{eq::castillo}
\begin{array}{rcl}
c_1p^{-c_2}\le&\frac{\pi(\A)}{\pi(\A-1)}&\le c_3p^{-c_4},
\end{array}
\end{equation}
for positive $c_i$. If we assumed $p=O(n)$, then \cref{eq::castillo} is a very strong penalization, and the consistencies have been proved in the papers adopted the sparse prior. \Cref{eq::castillo} is also equivalent to  $\pi(\gamma_j=1)=O(p^{-c_i})$, which is the prior inclusion probability of adding one generic covariate. From this perspective, it is obvious that the prior inclusion probability only depends on the size of the full model without any consideration of the size of the current testing model, which is less adjustable or flexible for a variety of situations. Also, \cref{eq::castillo} lacks a properly realistic explanation to the practitioners at face value. Last but not the least, when comparing the growing true model to a finite model nested in the true model, $\frac{\pi\lrp{M_A}}{\pi\lrp{M_T}}$ in \cref{eq::summation}   increases in an exponential rate of the sample size, which   imposes harsh and unrealistic conditions on  design matrix and signal noise ratio to achieve model selection consistency.\\

We propose to use a truncated Poisson prior on the size of models, such that 
\begin{equation}
\begin{array}{rcl}
\pi(\A)&=&\frac{\lambda^{\A}\exp{-\lambda}}{\A!},
\end{array}
\end{equation}
where $\lambda$ is the rate parameter, which can be tuned by practitioners. The derivation of the truncated Poisson prior comes from using the self-similarity property for model spaces with finite $p$ and letting $p\to\infty$ with an easy proof in \cite{womack2015model}. The prior inclusion probability of the truncated Poisson distribution is only associated with the testing model size $\A$, which is more adjustable than the sparsity prior. When dealing with the overfitted models, the truncated Poisson prior provides a well-designed penalty to convert the summation of the Bayes factors generated by the models with same size to its arithmetic average, hence the convergence is readdressed by the probabilistic property of the average of the Bayes factor instead of the overly-strong dimensional penalty such as the sparsity prior. Also, the average of the Bayes factor is an average of the random variables generated by randomly choosing the extraneous predictors, which leads to an interesting application of the stable law and the domain of attraction. We will give a more detailed explanation in Section 3. 
\subsection{Assumptions}
The assumptions that we will make in this section are crucial to the proof of the consistency. These assumptions reflect the integrity of the problem, i.e., the model with too strict assumptions would lost practical usage to real-world data yet the model without any assumptions could not be proved to attain the consistency. 

\begin{assumption}[Conditions on the design matrix]
\label{am::design_matrix}~\\
\begin{itemize}
 \item[(i)] Assume all the columns of the design matrix are standardized, such as 
\begin{equation}
||X_i||_2^2=n\text{, for all } ~i\in\lrc{1,\ldots,\F}. 
\end{equation}

\item[(ii)]There exists a positive constant $\zeta_{min}$, such that 
\begin{equation}
\begin{array}{c}
\stackbin[A\in\lrc{1,\ldots,p}]{}{\text{min}}\nu_{\text{min}}\lrp{\frac{1}{n}X_A^TX_A}\ge\zeta_{min}
\end{array}
\end{equation}
where $\nu(H)$ is the eigenvalue of a symmetric matrix $H$.
\item[(iii)] We assume a standardized errors constraining assumption, such as, 
\begin{equation}
\begin{array}{rcl}
\pmb{E}\lrb{\stackbin[k\notin B]{}{\text{max}}\stackbin[B\subset F;B\supset T]{}{\text{max}}\lrp{\frac{|<(I-H_{B})x_k,Z>|}{\sqrt{n}}}}\le\ff{\sqrt{\ff{\zeta_{min}}*\log{\log{n}}}},
\end{array}
\end{equation}
where $Z\sim N(0,\mathit{I}_{n\times n})$, and $H_B$ is the projection matrix generated by the design matrix $X_B$.
\end{itemize}
\end{assumption}
\begin{assumption}[Condition on the mean structure]
\label{am::finite_mean}
Given $X_T$, $\beta_T$, and $\sigma_T^2$, assume
\begin{equation}
\lim{\frac{||X_T\beta_T||^2}{n\sigma_T^2}}=C_1
\end{equation}
where $C_1$ is a finite positive constant.  
\end{assumption}
\begin{assumption}[Conditions on the growing rates] 
\label{am::growing_rate}
The growing rate of the full model can be only as fast as a fraction of the sample size, i.e., $\lim\frac{\F}{n}=f<1$.
\end{assumption}

\begin{assumption}[Condition on the minimum signal] Let $\beta_{min}$ denote $\stackbin[i\in T]{}{\text{min}}|\beta_i|$, then
we assume that $\tau_T||\beta_T||_2^2=O(1)$, which implies $\tau_T\beta_{min}^2=\Theta(\frac{1}{\T})$. Specifically, the minimum of the true coefficients is bounded above and below as 
\label{am::min_signal}
\begin{equation}
\frac{C_2}{\T}\le\tau_T\beta_{min}^2\le \frac{C_3}{\T}
\end{equation}
where $C_2$ is a finite constant bounded above by zero, and $C_3$ is a finite constant implied by \Cref{am::finite_mean}.
\end{assumption}
\begin{remark}
In \cref{am::design_matrix}, $(i)$ and $(ii)$ are typically  mild conditions that appear in many literatures considering the consistency, e.g.,  \cite{narisetty2014bayesian} \cite{womack2015model} \cite{yang2016computational} \cite{shang2011consistency}. $(ii)$ in \cref{am::design_matrix} guarantees that each potentially testing model is not too close to be distinguish to each other. $(iii)$
is to constrain the projection of the standardized errors on the space generated by any extraneous predictors excluding the space generated by the design matrix nesting the true model. This assumption is also a mild condition just to prevent any extreme behavior from the overfitted model and their extraneous predictors. 
\\

\Cref{am::finite_mean} is a reasonable condition putting onto the true mean structure, which surprisedly was not considered by many literatures explicitly. Without \cref{am::finite_mean}, under some circumstances, the true mean would grow to infinity with the sample size, which doesn't make any sense. \Cref{am::growing_rate} assumes the growing rate of the full model is linear to the sample size. \Cref{am::min_signal} is related to the information-theoretical capacity of the model to recovery the exact support  \cite{wainwright2009information}. As the growing rate of the true model hitting the limit, i.e., $\frac{n}{\log{n}}$, the minimum of the signal has to correspond to such harsh condition to separate itself from the noise. The reason for this rate is shown in \cref{prop::true_model}.
\end{remark}

\subsection{Result}
\begin{proposition}
\label{prop::true_model}
Suppose that \Cref{am::design_matrix}, \Cref{am::finite_mean}, \Cref{am::growing_rate},  and \Cref{am::min_signal} hold, the fastest growing rate of the true model that the procedure can tolerate is $\mathcal{O}\lrp{\frac{n}{\log{n}}}$.
\begin{proof}
See Section 5.
\end{proof}
\end{proposition}
Here are the main theorem of this paper. 
 \begin{theorem}
 \label{tm::main}
 Suppose that \Cref{am::design_matrix}, \Cref{am::finite_mean}, \Cref{am::growing_rate},  and \Cref{am::min_signal} hold. 
Suppose $\T=\frac{n}{\log{n}}$. If
\begin{equation}
C_2>\frac{10}{\zeta_{min}},
\end{equation}
then, we have 
\begin{equation}
\stackbin[n\to\infty]{}{\text{lim}}\text{Pr}\lrp{M_T|y}=1
\end{equation}
with probability at least $1-\lrb{\log{n}}^{-\frac{1}{4}}$. 
\begin{proof}
See Section 5. 
\end{proof}
\end{theorem}
The constant $C_2$ characterizes the relationship between the minimums of the eigenvalues and the signals, which all hinges on the growing rate of the true model.   The probability in \cref{tm::main} is determined by the worst scenario, where the testing models are  the overfitted models with only one extraneous predictor. The reason behind such a slowly convergent rate is due the fast growing rate of  the true model. \\

It is reasonable to consider the scenario when the growing rate of the true model is $\mathcal{O}(n^{d})$ for $d<1$. It can be shown that the relatively strict assumptions can be loosen, and the probability associated with the consistency would be faster on a different order. 
\section{Stable law}
One interesting encounter from the proof is an application of the stable law and the domain of attraction, when dealing with the overfitted models. Let $\mathcal{M}_c$ denote the model set of  overfitted models with number of $c$ extraneous predictors, such that
\begin{equation}
\label{df::model_set}
\begin{array}{rcl}
\mathcal{M}_c&:=&\lrc{M_A;A\supset T, ~\A-\T=c}.
\end{array}
\end{equation}
Then, it is straight forward to show 
\begin{equation}
\begin{array}{rcl}
\stackbin[M_A\in\mathcal{M}_c]{}{\sum}\frac{m\lrp{y|M_A}}{m\lrp{y|M_T}}\frac{m\lrp{M_A}}{m\lrp{M_T}}=\frac{\overline{BF_{A:T}}}{c!}\lambda^{c},
\end{array}
\end{equation}
where  $\overline{BF_{A:T}}$ is the arithmetic average of the Bayes factor over the model set $\mathcal{M}_c$. As $\F-\T\to\infty$, each Bayes factor can be seen as a random variable which is randomly chosen from the set $\mathcal{M}_c$. Furthermore, as shown in \cref{eq::bayes_factor}, the Bayes factor can be well approximated as
\begin{equation}
\label{eq::approx}
\begin{array}{rcl}
\text{\cref{eq::bayes_factor}}&\approx& \exp{\ff{\eta_i}}\frac{\Gamma\lrp{\ff{\A}}\Gamma\lrp{\ff{n-\A}} }{\Gamma\lrp{\ff{\T}}\Gamma\lrp{\ff{n-\T}}},
\end{array}
\end{equation}
where $\eta_i\sim\chi^2(c)$ for $i=1,\ldots,{\F-\T \choose c}$, and the covergent rate of the fractions associated with the terms of the Gamma functions is $\mathcal{O}\lrp{\log{n}^{-\ff{c}}}$. Hence, the average over the Bayes factor is equivalent to the average over the random variable $\exp{\ff{\eta_i}}$, which is generated by  models $M_A\in\mathcal{M}_c$. However, a simple integration can show that the r.v. $\delta_i=\exp{\ff{\eta_i}}$ doesn't have any finite moment. Hence, to show the convergence of the average, the stable law is required. The next theorem concludes the convergence of the average for i.i.d. situation. 
\begin{theorem}
Suppose $\delta_i=\exp{\ff{\eta_i}}$ and $\eta_i\stackbin[i.i.d.]{}{\sim}\chi^2(c)$, where $c$ is a finite positive integer and $i=1,\ldots,m$. Then given the constants $a_m=m*\frac{\log{m}^{\ff{c}-1}}{\Gamma\lrp{\ff{c}}}$ and $b_m=\frac{\log{a_m}^{\ff{c}}}{\Gamma\lrp{\ff{c}+1}}$, we have
\begin{equation}
\begin{array}{rcl}
\lim\frac{\sum_{i=1}^m\delta_i-mb_m}{a_m}&=&S_\alpha(1,1,0),
\end{array}
\end{equation}
where $S_\alpha(1,1,0)$ is a Cauchy random variable. 
\end{theorem}
This theorem reveals the fact that the average of the Bayes factor converges to a Cauchy r.v. and a slowing varying function associated with the total number of the models in the set $\mathcal{M}_c$. Since the convergent power provided by the Bayes factor is only $\mathcal{O}\lrp{\log{n}^{-\ff{c}}}$, the corollary next shows that it is impossible to show the convergence for the overfitted models without assumptions controlling the behavior of the errors. 
\begin{corollary}
Given a dimensional difference $c$ and the corresponding model set $\mathcal{M}_c$ defined in \cref{df::model_set}, we have 
\begin{equation}
\begin{array}{rcl}
\lim\stackbin[M_A\in\mathcal{M}_c]{}{\sum}\frac{m\lrp{y|M_A}}{m\lrp{y|M_T}}\frac{m\lrp{M_A}}{m\lrp{M_T}}&\neq&0.
\end{array}
\end{equation}
\begin{proof}
Let $m={\F-\T\choose c}\asymp n^{c}$ denote the number of the models in $\mathcal{M}_c$. We have seen that the summation above can be represented the average of the r.v. $\delta_i$, and the Gamma functions in \cref{eq::approx} only has the rate of $\mathcal{O}\lrp{\log{n}^{-\ff{c}}}$. Suppose $\delta_i$s are i.i.d., hence 
  \begin{equation}
\begin{array}{rcl}
\stackbin[M_A\in\mathcal{M}_c]{}{\sum}\frac{m\lrp{y|M_A}}{m\lrp{y|M_T}}\frac{m\lrp{M_A}}{m\lrp{M_T}}&\approx& \overline{\delta_i}*\log{n}^{-\ff{c}}\\
&\to&\frac{\log{m}^{\ff{c}-1}}{\Gamma\lrp{\ff{c}}}\lrb{S_\alpha(1,1,0)+\frac{\log{m}}{\ff{c}}}\log{n}^{-\ff{c}}\\
&\to&h(c)\neq0,
\end{array}
\end{equation}
where $h(c)$ is a function of the difference $c$, and not equal to 0 if $c$ is a finite positive integer. 
\end{proof}
\end{corollary}
This corollary shows that the consistency requires extra assumptions, e.g., (iii) in \cref{am::design_matrix}, to control the projection of the errors on the subspace generated by the non-true predictors, when the growing rate of the true model has reached the limit, i.e., $\frac{n}{\log{n}}$. If the growing rate of the true model is slower than the limit, then the convergent rate provided by the Bayes factor will overcome the slowly varying function in $a_m$ and $b_m$.
\section{Conclusion}
In this paper, the model selection problem of Gaussian regression in high-dimensional regime has been studied. The mixture of $g$-prior and the truncated Poisson prior have been proposed to show the consistency of model selection. We presented the consistency theorem under the situation of extreme growing rate. It showed that when the limit of the growing rate of the true model is approached, the consistency theorem requires more strict assumptions. The stable law has been applied to make the argument of the unidentifiable situation for the overfitted models, when there is no extra assumptions on the extraneous predictors and the growing rate of the true model reaches the limit. \\

For the next step, it is interesting to further investigate  the stable law under the dependent situation due to the possible  dependency structure between the extraneous covariates. Another direction would be to discuss the situation when the size of the full model is greater than the sample size. The authors in  \cite{yang2016computational} simply truncated the model space by assigning zero probability to the models whose size is greater than the sample size. We consider to adopt the PCA method to reduce the dimension as the first step of the model selection. 
\section{Appendix}
\subsection{Lemmas}
Let $H_A$ denote the projection matrix onto the span of $\lrc{X_A,A\in F}$. 
\begin{lemma}
\label{lm::non-ex}
A projection is a non-expensive mapping\cite{yang2016computational}. In the words, for any column $x_i$, $i\in F$, 
\begin{equation}
\begin{array}{ccc}
||H_Ax_i||&\le&||x_i||
\end{array}
\end{equation}
\end{lemma}
\begin{proof}
\begin{equation}
\begin{array}{ccc}
||H_Ax_i||^2&=&\langle H_Ax_i,H_Ax_i\rangle\\
&=&x_i'H_Ax_i\\
&=&\lrl{x_i,H_Ax_i}\\
&\stackbin[]{(i)}{\le}&||x_i| ||H_Ax_i||
\end{array}
\end{equation}
(i) is by the Cauchy-Schwarz inequality. 
\end{proof}
\begin{lemma}
\label{lm::subspace}
For any column $x_i$, $i\in F$,
\begin{equation}
\frac{1}{\sqrt{n}}|| \lrp{{\mathbf I}-H_A}x_i ||\le1
\end{equation}
\begin{proof}
Direct result from use of \Cref{lm::non-ex} and $(i)$ in \Cref{am::design_matrix}.
\end{proof}
\end{lemma}
\begin{lemma}
\label{lm::sub_eigen}
Under $(ii)$ \Cref{am::design_matrix}, for any pair $A\in F$, $A'\in F$, $A\subset A'$, we have
\begin{equation}
\begin{array}{c}
\nu_{\text{min}}\lrp{\frac{1}{n}X'_{A'/A}\lrp{{\mathbf I}-H_A}X_{A'/A}}\ge \zeta_{min}
\end{array}
\end{equation}
\begin{proof}
W.t.l.g., assume $X_{A'}=\lrc{x_A,x_{A'/A}}$, then by the formula of the blockwise inversion, one can show that the lower right corner of the matrix $\lrp{\frac{1}{n}{X_{A'}}'X_{A'}}^{-1}$ is $\lrp{\frac{1}{n}X'_{A'/A}\lrp{{\mathbf I}-H_A}X_{A'/A}}^{-1}$. The rest follows. 
\end{proof}
\end{lemma}
\begin{lemma}
For any $n\times 1$ vector $a$ and a generic symmetric matrix $A$, we have
\begin{equation}
\zeta_{max}\ge\frac{a'Aa}{a'a}\ge \zeta_{min}
\end{equation}
where $\zeta_{min}$ and $\zeta_{max}$ are the smallest and largest eigenvalues of $A$.
\end{lemma}
\begin{proof}
See \cite{schott2016matrix}.
\end{proof}

\begin{lemma}
\label{lm::under_non-centra}
Suppose $M_A$ is in the model set $\mathcal{M}_{c,k}:=\lrc{M_A: \T-\A=c, A/T=k}$, then it can be shown that  
\begin{equation}
\begin{array}{rcl}
\lambda_{A\cup T-A}&\ge&(c+k)\beta_{min}^2\tau_T\zeta_{min}n
\end{array}
\end{equation}
 \end{lemma}
\begin{proof}
Notice that $T/A=c+k$.
 \begin{equation}
\begin{array}{rcl}
\lambda_{A\cup T-A}&=&\tau_T\lrp{\beta_T'X_T'(H_{A\cup T}-H_{A})X_T\beta_T}\\
&=&\tau_T\lrp{\beta_T'X_T'(I-H_A)X_T\beta_T}\\
&=&\tau_T\lrp{\beta_{T/A}'X_{T/A}'(I-H_A)X_{T/A}\beta_{T/A}}\\
&\ge&\tau_T\zeta_{min}n||\beta_{T/A}||_2^2\\
&\ge&(c+k)\beta_{min}^2\tau_T\zeta_{min}n
\end{array}
\end{equation}
\end{proof}
\begin{remark}
there is a question. 
\end{remark}
 \subsection{Proof of \Cref{tm::main}}
 \subsubsection{$\A<\T$}
 
We begin by proving the consistency under the situation $\mathcal{M}_{c,k,s}:=\lrc{M_A: \T-\A=c, A/T=k}$. First, notice that there are ${\T\choose c+k}{\F-\T\choose k}$ models associated with $\mathcal{M}_{c,k}$ for specific $c$ and $k$. The Bayes factor for any $M_A\in \mathcal{M}_{c,k}$ is 
 \begin{equation}
\begin{array}{rcl}
BF_{A:T}&=&\lrp{\frac{1-R_T^2}{1-R_A^2}}^{\ff{n-\T}}\lrp{1-R_A^2}^{-\ff{c}}\frac{\Gamma(\ff{\A})\Gamma(\ff{n-\A})}{\Gamma(\ff{\T})\Gamma(\ff{n-\T})}
\end{array}
\end{equation}
The terms associated with the coefficients of determination can be shown as
\begin{equation}
\begin{array}{rcl}
\lrp{\frac{1-R_T^2}{1-R_A^2}}&\to&\frac{n-\T}{n-\T+c+\lambda_{H_{A\cup T}-H_{T}}}\\
&\le&\frac{n-\T}{n-\T+\lambda_{H_{A\cup T}-H_{T}}}\\
&\le&1-\frac{\lambda_{H_{A\cup T}-H_{T}}}{n-\T+\lambda_{H_{A\cup T}-H_{T}}}\\
&\le&1-\text{min}\lrc{\frac{1}{2},\frac{\lambda_{H_{A\cup T}-H_{T}}}{2(n-\T)}}\\
&{\le}&1-\frac{\lambda_{H_{A\cup T}-H_{T}}}{2(n-\T)} \text{\qquad; with sufficiently large $n$}\\
&\stackbin[]{(i)}{\le}&1-\frac{(c+k)\beta_{min}^2\tau_T\zeta_{min}n}{2(n-\T)}
\end{array}
\end{equation}
where $(i)$ is due to \Cref{lm::under_non-centra}. It can also be shown that
\begin{equation}
\begin{array}{rcl}
\lrp{1-R_A^2}^{-\ff{c}}&\to&\lrp{\frac{n+nC_1}{n-\A+\lambda_{I-H_A}}}^{\ff{c}}\\
&\le&\lrp{\frac{n+nC_1}{n-\A+(c+k)\beta_{min}^2\tau_T\zeta_{min}n}}^{\ff{c}}\\
\end{array}
\end{equation}
Let $\lim\frac{n+nC_1}{n-\A+(c+k)\beta_{min}^2\tau_T\zeta_{min}n}=C_4=O(1)$. Then, 
\begin{equation}
\begin{array}{rcl}
\frac{\lrp{1-R_T^2}^{\ff{n-\T}}}{\lrp{1-R_A^2}^{\ff{n-\A}}}&=&\lrp{\frac{1-R_T^2}{1-R_A^2}}^{\ff{n-\T}}\lrp{1-R_A^2}^{-\ff{c}}\\
&\le&\exp{-\frac{(c+k)\beta_{min}^2\tau_T\zeta_{min}n}{4}+\ff{c}\log{C_4}}
\end{array}
\end{equation}
The summand then can be represented as
\begin{equation}
\begin{array}{rcl}
\stackbin[M_A\in\mathcal{M}_{c,k,s}]{}{\sum} BF_{A:T}PO_{A:T}&\le&{\T\choose c+k}{\F-\T\choose k}\frac{(\F-\A)!}{(\F-\T)!}\frac{\Gamma(\ff{\A})\Gamma(\ff{n-\A})}{\Gamma(\ff{\T})\Gamma(\ff{n-\T})}\exp{-\frac{(c+k)\beta_{min}^2\tau_T\zeta_{min}n}{4}+\ff{c}\log{C_4}} \\
&\le&\exp{-\frac{(c+k)\beta_{min}^2\tau_T\zeta_{min}n}{4}+\frac{c}{2}+\frac{5}{2}c\log{n}+2k\log{n}+\ff{c}\log{C_4}}\\
&\asymp&\exp{-\frac{(c+k)\beta_{min}^2\tau_T\zeta_{min}n}{4}+\frac{5}{2}c\log{n}+2k\log{n}}.
\end{array}
\end{equation} 
 If $\T=\frac{n}{\log{n}}$, then the consistency requires that $C_2>\frac{10}{\zeta_{min}}$. On the other hand, suppose that $\T=n^{d}$ for $0\le d<1$. Then $\beta_{min}^2\tau_T\zeta_{min}n=O(n^{1-d})$, which is growing faster than $\log{n}$. This implies that the $\zeta_{min}$ can converge to zero with a mild speed under such circumstance. 
 \subsubsection{$\A>\T$ and $A\cap T\neq T$}
 Under this circumstance, the dimensional penalty helps the procedure to pick up the true model comparing to the previous situation. Define the model set as \\ $\mathcal{M}_{c,k,l}:=\lrc{M_A: \A-\T=c, T/A=k}$, then the total number of models given any $c$ and $k$ is ${\T\choose k}*{\F-\T\choose k+c}$. The terms in the summand excluding the terms associated with the coefficient of determination can be shown as
\begin{equation}
\label{eq::DP_L}
\begin{array}{rcl}
{\T\choose k}{\F-\T\choose k+c}\frac{(\F-\A)!}{(\F-\T)!}\frac{\Gamma(\ff{\A})\Gamma(\ff{n-\A})}{\Gamma(\ff{\T})\Gamma(\ff{n-\T})}&\preccurlyeq&\exp{2k\log{n}+2k-c}
\end{array}
\end{equation}
It is straight forward to show that 
\begin{equation}
\label{eq::CoD_L}
\begin{array}{rcl}
\frac{\lrp{1-R_T^2}^{\ff{n-\T}}}{\lrp{1-R_A^2}^{\ff{n-\A}}}&=&\lrp{\frac{1-R_T^2}{1-R_A^2}}^{\ff{n-\T}}\lrp{1-R_A^2}^{\ff{c}}\\
&\to& \lrp{\frac{n-\T}{n-\T-c+\lambda_{H_{A\cup T}-H_A}}}^{\ff{n-\T}}\lrp{\frac{n-\A+\lambda_{I-H_A}}{n+C_1n}}^{\ff{c}}\\
&\le& \lrp{1-\frac{\lambda_{H_{A\cup T}-H_A}-c}{n-\T-c+\lambda_{H_{A\cup T}-H_A}}}^{\ff{n-\T}}\lrp{\frac{n-\A+n\zeta_{min}||\beta_{T/A}||_2^2\tau_T}{n+C_1n}}^{\ff{c}}\\
&\le&\exp{-\frac{k\zeta_{min}\log{n}C_2-c}{4}+\ff{c}\log{C_5}},
\end{array}
\end{equation}
 where $\lim\frac{n-\A+n\zeta_{min}||\beta_{T/A}||_2^2\tau_T}{n+C_1n}=C_5\le1$. Combining \Cref{eq::DP_L} and \Cref{eq::CoD_L}, it gives us
 \begin{equation}
\begin{array}{rcl}
\stackbin[M_A\in\mathcal{M}_{c,k,l}]{}{\sum} BF_{A:T}PO_{A:T}&\le&\exp{2k\log{n}+2k-c}\exp{-\frac{k\zeta_{min}\log{n}C_2-c}{4}+\ff{c}\log{C_5}}
\end{array}
\end{equation} 
 Given the condition $C_2>\frac{10}{\zeta_min}$ in the \Cref{tm::under-fitted}, the consistency holds for this situation.  
 \subsection{Overfitted Model}
Suppose that $\T=\frac{tn}{\log{n}}$, for $0<t<1$. 
  Define the overfitted model set\\ $\mathcal{M}_{c}:=\lrc{M_A:A\supset T, \A-\T=c}$ associated with a given finite $c$. Notice that the summation over the model set $\mathcal{M}_c$ is
    \begin{equation}
\begin{array}{rcl}
\stackbin[M_A\in\mathcal{M}_{c}]{}{\sum} BF_{A:T}PO_{A:T}&=&\frac{\overline{BF}}{c!}\lambda^{\ff{c}}. 
\end{array}
\end{equation} 
The individual Bayes factor is asymptotically equivalent to 
    \begin{equation}
\begin{array}{rcl}
BF_{A:T}\asymp\exp{\frac{\xi_A}{2}-\ff{c}\log{\log{n}}-2*\log{t}},
\end{array}
\end{equation} 
where $\xi_A=\tau_TY'(H_A-H_T)Y$ and $t=1$ for now. To show $\overline{BF}\to0$, it suffices to show that 
    \begin{equation}
    \label{eq::max_BF}
\begin{array}{rcl}
\exp{\lrb{\frac{\stackbin[M_A\in\mathcal{M}_c]{}{\text{max}}\lrp{\xi_A}}{2}}-\ff{c}\log{\log{n}}}\to0. 
\end{array}
\end{equation} 
This can be achieved by showing the probability of the  event defined as $\Cref{eq::max_BF}$ going to 1, such as 
    \begin{equation}
\begin{array}{rcl}
\Prob{\exp{\lrb{\frac{\stackbin[M_A\in\mathcal{M}_c]{}{\text{max}}\lrp{\xi_A}}{2}}-\ff{c}\log{\log{n}}}\ge\epsilon_n }&=&\Prob{{\lrb{\frac{\stackbin[M_A\in\mathcal{M}_c]{}{\text{max}}\lrp{\xi_A}}{2}}-\ff{c}\log{\log{n}}}\ge\log{\epsilon_n} }\\
&=&\Prob{{\lrb{{\stackbin[M_A\in\mathcal{M}_c]{}{\text{max}}\lrp{\xi_A}}}\ge\ff{c}*\log{\log{n}}} }\\
&\le&\exp{-w'},
\end{array}
\end{equation} 
where $\epsilon_n=\lrb{\log{n}}^{-\frac{c}{4}}\stackbin[n\to\infty]{}{\to}0$, and $w'\stackbin[n\to\infty]{}{\to}\infty$ to guarantee the convergence. 
Notice that 
    \begin{equation}
\begin{array}{rcl}
\stackbin[M_A\in\mathcal{M}_c]{}{\text{max}}\xi_A&=&\stackbin[M_A\in\mathcal{M}_c]{}{\text{max}}\tau_TY'(H_A-H_T)Y\\
&=&\stackbin[M_A\in\mathcal{M}_c]{}{\text{max}}\tau_T\pmb{\epsilon}'(H_A-H_T)\pmb{\epsilon}\\
&=&\stackbin[M_A\in\mathcal{M}_c]{}{\text{max}}\tau_T\sum_{i=1}^{i=c}\pmb{\epsilon}'(H_i-H_{i-1})\pmb{\epsilon}\\
&=&\stackbin[M_A\in\mathcal{M}_c]{}{\text{max}}\tau_T\sum_{i=1}^{i=c}\pmb{\epsilon}'\frac{(I-H_{i-1})x_ix_i'(I-H_{i-1})}{x_i'(I-H_{i-1})x_i}\pmb{\epsilon}\\
&\le&\stackbin[M_A\in\mathcal{M}_c]{}{\text{max}}\tau_T\sum_{i=1}^{i=c}\pmb{\epsilon}'\frac{(I-H_{i-1})x_ix_i'(I-H_{i-1})}{n\zeta_{min}}\pmb{\epsilon}\\
&\le&c*\tau_T\stackbin[k\notin B]{}{\text{max}}\stackbin[B\subset F;B\supset T]{}{\text{max}}\pmb{\epsilon}'\frac{(I-H_{B})x_kx_k'(I-H_{B})}{n\zeta_{min}}\pmb{\epsilon}\\
&=&\frac{c*\tau_T}{\zeta_{min}}\stackbin[k\notin B]{}{\text{max}}\stackbin[B\subset F;B\supset T]{}{\text{max}}\lrp{\frac{<(I-H_{B})x_k,\pmb{\epsilon}>}{\sqrt{n}}}^2,
\end{array}
\end{equation} 
which means that it suffices to show
\begin{equation}
\begin{array}{rcl}
\Prob{{\lrb{{\stackbin[M_A\in\mathcal{M}_c]{}{\text{max}}\lrp{\xi_A}}}\ge\ff{c}*\log{\log{n}}} }&\le&\Prob{\frac{c*\tau_T}{\zeta_{min}}\stackbin[k\notin B]{}{\text{max}}\stackbin[B\subset F;B\supset T]{}{\text{max}}\lrp{\frac{<(I-H_{B})x_k,\pmb{\epsilon}>}{\sqrt{n}}}^2\ge\ff{c}*\log{\log{n}}}\\
&=&\Prob{\stackbin[k\notin B]{}{\text{max}}\stackbin[B\subset F;B\supset T]{}{\text{max}}\lrp{\frac{|<(I-H_{B})x_k,\frac{\pmb{\epsilon}}{\sigma}>|}{\sqrt{n}}}\ge\sqrt{\ff{\zeta_{min}}*\log{\log{n}}}}\\
&\le&\exp{-w'}.
\end{array}
\end{equation}
Define a function 
\begin{equation}
\begin{array}{rcl}
V(Z)&:=&\stackbin[k\notin B]{}{\text{max}}\stackbin[B\subset F;B\supset T]{}{\text{max}}\lrp{\frac{|<(I-H_{B})x_k,Z>|}{\sqrt{n}}}.
\end{array}
\end{equation}
For any $Z$, $Z_1\in \mathcal{R}^n\sim N(0,I_{n\times n})$, it can be shown
\begin{equation}
\label{eq::Lipschitz}
\begin{array}{rcl}
|V(Z)-V(Z_1)|&=&|\stackbin[k\notin B]{}{\text{max}}\stackbin[B\subset F;B\supset T]{}{\text{max}}\lrp{\frac{|<(I-H_{B})x_k,Z>|}{\sqrt{n}}}-\stackbin[k\notin B]{}{\text{max}}\stackbin[B\subset F;B\supset T]{}{\text{max}}\lrp{\frac{|<(I-H_{B})x_k,Z_1>|}{\sqrt{n}}}|\\
&\le&\stackbin[k\notin B]{}{\text{max}}\stackbin[B\subset F;B\supset T]{}{\text{max}}|<\frac{(I-H_{B})x_k}{\sqrt{n}},Z-Z_1>|\\
&\le&\stackbin[k\notin B]{}{\text{max}}\stackbin[B\subset F;B\supset T]{}{\text{max}}\frac{||(I-H_{B})x_k||_2}{\sqrt{n}}||Z-Z_1||_2\\
&\stackbin[]{(i)}{\le}&||Z-Z_1||_2,
\end{array}
\end{equation}
where $(i)$ is due to the normalization assumption. \Cref{eq::Lipschitz} shows that $V(Z)$ is a Lipschitz function, and $||V(\bullet)||_{Lip}=1$. By the Theorem 5.2.2 in \cite{vershynin2016high}, and also can be found in \cite{ledoux2001concentration}, we have
\begin{equation}
\begin{array}{rcl}
\Prob{V(Z)\ge E\lrb{V(Z)}+t}\le\exp{-\ff{t^2}}.
\end{array}
\end{equation}
Now we bound the expectation of $V(Z)$ by introducing the following assumptions
\begin{assumption}
\begin{equation}
\begin{array}{rcl}
\pmb{E}\lrb{\stackbin[k\notin B]{}{\text{max}}\stackbin[B\subset F;B\supset T]{}{\text{max}}\lrp{\frac{|<(I-H_{B})x_k,Z>|}{\sqrt{n}}}}\le\ff{\sqrt{\ff{\zeta_{min}}*\log{\log{n}}}}.
\end{array}
\end{equation}
\end{assumption}
This assumption ensures that 
\begin{equation}
\begin{array}{rcl}
\Prob{V(Z)\ge2\ff{\sqrt{\ff{\zeta_{min}}*\log{\log{n}}} }}&\le&\Prob{V(Z)\ge\pmb{E}\lrb{V(Z)}+\ff{\sqrt{\ff{\zeta_{min}}*\log{\log{n}}} }}\\
&\le&\exp{-\frac{{\zeta_{min}}*\log{\log{n}}}{16}}\\
&=&\lrp{\log{n}}^{-\frac{\zeta_min}{16}}
\end{array}
\end{equation}
This completes the convergence as 
\begin{equation}
\begin{array}{rcl}
&&\Prob{\exp{\lrb{\frac{\stackbin[M_A\in\mathcal{M}_c]{}{\text{max}}\lrp{\xi_A}}{2}}-\ff{c}\log{\log{n}}}\ge\epsilon_n }\\
&\le&\Prob{\stackbin[k\notin B]{}{\text{max}}\stackbin[B\subset F;B\supset T]{}{\text{max}}\lrp{\frac{|<(I-H_{B})x_k,\frac{\pmb{\epsilon}}{\sigma}>|}{\sqrt{n}}}\ge\sqrt{\ff{\zeta_{min}}*\log{\log{n}}}}\\
&\le&\Prob{\stackbin[k\notin B]{}{\text{max}}\stackbin[B\subset F;B\supset T]{}{\text{max}}\lrp{\frac{|<(I-H_{B})x_k,\frac{\pmb{\epsilon}}{\sigma}>|}{\sqrt{n}}}
\ge\pmb{E}\lrb{\stackbin[k\notin B]{}{\text{max}}\stackbin[B\subset F;B\supset T]{}{\text{max}}\lrp{\frac{|<(I-H_{B})x_k,\frac{\pmb{\epsilon}}{\sigma}>|}{\sqrt{n}}}}+\ff{\sqrt{\ff{\zeta_{min}}*\log{\log{n}}} }}\\
&\le&\lrp{\log{n}}^{-\frac{\zeta_min}{16}},
\end{array}
\end{equation}
where $\epsilon_n=\lrb{\log{n}}^{-\frac{c}{4}}$. 
   \bibliographystyle{acm}

\end{document}